\documentclass{amsart}

\usepackage{amsmath}
\usepackage{amsthm}
\usepackage{amsfonts}
\usepackage{amssymb}
\usepackage{enumerate}

\newcommand{\E}{\mathbb{E}}
\newcommand{\Prob}{\mathbb{P}}
\newcommand{\C}{\mathbb{C}}

\newcommand{\var}{\mathrm{Var}}

\newcommand{\tr}{\mathrm{Tr}}

\renewcommand\Im{\operatorname{Im}}

\theoremstyle{plain}
  \newtheorem{theorem}{Theorem}

  \newtheorem{lemma}[theorem]{Lemma}

\theoremstyle{definition}
  \newtheorem{definition}[theorem]{Definition}
  \newtheorem{example}[theorem]{Example}
  \newtheorem{remark}[theorem]{Remark}

\begin{document}
\title[A Note on the Marchenko-Pastur law]{A Note on the Marchenko-Pastur Law for a class of random matrices with dependent entries}
\author[S. O'Rourke]{Sean O'Rourke} 
\address{Department of Mathematics, Rutgers, Piscataway, NJ 08854  }
\email{sdo21@math.rutgers.edu}

\begin{abstract}
We consider a class of real random matrices with dependent entries and show that the limiting empirical spectral distribution is given by the Marchenko-Pastur law.  Additionally, we establish a rate of convergence of the expected empirical spectral distribution.  
\end{abstract}

\maketitle

\section{Introduction and Main Results}

Suppose $M_n$ is a $n \times n$ matrix with real eigenvalues $\lambda_1, \lambda_2, \ldots, \lambda_n$.  Then the empirical spectral distribution (ESD) of the matrix $M_n$ is defined by 
$$ F^{M_n}(x) := \frac{ \# \left\{ 1 \leq i \leq n : \lambda_i \leq x \right\} }{n} .$$
We will be interested in the case when $M_n := \frac{1}{n} A_n A_n^\mathrm{T}$ and $A_n$ is an $n \times N$ real random matrix.  

If the entries of $A_n$ are i.i.d. random variables with zero mean and variance one, we call $M_n$ a sample covariance matrix.  There are many results concerning the limiting behavior of the spectral distribution of sample covariance matrices.  For example, Marchenko and Pastur (\cite{mp}) and Wachter (\cite{wa}) prove that the ESD $F^{\frac{1}{n} A_n A_n^\mathrm{T}}(x)$ converges to $F_c(x)$ provided that $N/n \rightarrow c \in (0, \infty)$, where $F_c$ is the distribution function for the Marchenko-Pastur law with parameter $c>0$.  That is, $F_c$ has density
$$ p_c(x) = \left\{
     \begin{array}{lr}
       \frac{ \sqrt{(x-a)(b-x)}}{2 \pi x} & : a \leq x \leq b, \\
       0 & :  \text{otherwise},
     \end{array}
   \right.
$$
and a point mass $1-c$ at the origin if $c < 1$, where $a = (1 -\sqrt{c})^2$ and $b = (1 + \sqrt{c})^2$.  The above convergence holds with probability $1$ (see for example \cite{B} and \cite[Chapter 3]{bai-book}).  

There are a number of results in which the independence assumption (on the entries of $A_n$) is weakened.  In the seminal paper by Marchenko and Pastur \cite{mp}, one considers independent rows rather than independent entries.  In \cite{yin}, Yin and Krishnaiah consider the case where the independent rows have a spherically symmetric distribution.  

More recently in 2006, Aubrun obtained the Marchenko-Pastur law for matrices with independent rows distributed uniformly on the $l_p^n$ balls, \cite{aubrun}.  This was generalized by Pajor and Pastur in \cite{PP} to matrices with independent rows distributed according to an arbitrary isotropic log-concave measure.  

In \cite{GT} and \cite{GT3}, G\"{o}tze and Tikhomirov study two classes of random matrices which generalize Wigner random matrices and sample covariance random matrices.   In particular, these matrices satisfy certain martingale-type conditions without any assumption on the independence of the entries.  In a similar setting, Adamczak studied a class of random matrices with uncorrelated entries in which each normalized row and normalized column converges to one in probability, \cite{A}.  

Other random matrix ensembles with dependent entries that have been studied include random Markov matrices with independent rows and doubly stochastic random matrices (see \cite{Chafai, bcc, cds} and references contained therein).  

In this note, we study a class of random matrices with dependent entries and show that the limiting empirical distribution of the eigenvalues is given by the Marchenko-Pastur law.  In particular, we consider a sequence of $n \times N$ random matrices $A_n$ with the following properties.

\begin{definition}[Condition {\bf C0}] \label{def:C0}
Let $\{A_n\}_{n \geq 1}$ be a sequence of $n \times N$ real random matrices where $N = N(n)$ and $c_n := N/n$.  We let $r_1^{(n)}, \ldots, r_n^{(n)}$ denote the rows of $A_n = (\zeta_{ij}^{(n)})_{1 \leq i \leq n, 1 \leq j \leq N}$ and define the $\sigma$-algebra associated to row $k$ as
$$ \mathcal{F}_k^{(n)} := \sigma( r_1^{(n)}, \ldots, r_{k-1}^{(n)}, r_{k+1}^{(n)}, \ldots, r_n^{(n)} ) $$
for all $k = 1, \ldots, n$.  Let $\E_k[ \cdot ]$ denote the conditional expectation with respect to the $\sigma$-algebra associated to row $k$.  We then say that the sequence $\{A_n\}_{n \geq 1}$ obeys condition {\bf C0} if the following hold:
\begin{enumerate}[(i)]
\item $\E_k[ \zeta_{ki}^{(n)} ] = 0$ for all $i,k,n$ \label{C0:uncor}
\item One has
$$ q_n := \sup_{k} \frac{1}{n} \sum_{i=1}^N \E|\E_k[ (\zeta_{ki}^{(n)})^2 ] - 1| = o(1) $$
 \label{C0:var}
\item One has 
$$ \sup_{k, i \neq j} |\E_k[ \zeta_{ki}^{(n)} \zeta_{kj}^{(n)}]| + \sup_{k,i,j \neq l} |\E_k[ (\zeta_{ki}^{(n)})^2 \zeta_{kj}^{(n)} \zeta_{kl}^{(n)}]| = O(n^{-1/2} \gamma_n) $$
a.s., where $\gamma_n \rightarrow 0$ as $n \rightarrow \infty$. \label{C0:2cor}
\item $\sup | \E_k[ \zeta_{ki}^{(n)} \zeta_{kj}^{(n)} \zeta_{kl}^{(n)} \zeta_{km}^{(n)}]| = O(n^{-1} \gamma_n)$ a.s where the supremum is over all $k$ and all $i,j,l,m$ distinct. \label{C0:4cor}
\item $\sup_{n,i,j} \E |\zeta_{ij}^{(n)}|^4 \leq M < \infty$ \label{C0:moment}
\item One has
$$ \rho_n := \sup_{k} \frac{1}{n^2} \sum_{1 \leq i,j \leq N} \E|\E_k[ (\zeta_{ki}^{(n)})^2 (\zeta_{kj}^{(n)})^2] - 1| = o(1). $$ \label{C0:row_variance}
\item There exists a non-negative integer sequence $\beta_n = o(\sqrt{n})$ such that $\sigma(r_{i_1}^{(n)}, \ldots, r_{i_k}^{(n)})$ and $\sigma(r_{j_1}^{(n)}, \ldots, r_{j_m}^{(n)})$ are independent $\sigma$-algebras whenever
$$ \min_{1 \leq l \leq k, 1 \leq p \leq m} |i_l - j_p| > \beta_n. $$ \label{C0:indep}
\end{enumerate}
\end{definition}

\begin{remark}
Condition \eqref{C0:uncor} implies that entries from different rows are uncorrelated while \eqref{C0:2cor} and \eqref{C0:4cor} allow for a weak correlation amongst entries in the same row.  Condition \eqref{C0:var} is a requirement on the variance of the entries and \eqref{C0:moment} is a moment assumption on the entries.  Condition \eqref{C0:row_variance} is of a technical nature.  In particular, \eqref{C0:row_variance} (along with \eqref{C0:var}) allows one to control terms of the form
$$ \sup_{k} \var \left( \frac{1}{n} |r_k^{(n)}|^2 \right) $$
where $|r_k^{(n)}|$ is the Euclidian norm of the vector $r_k^{(n)}$.  In words, condition \eqref{C0:indep} implies that rows, which are ``far enough apart,'' are independent.  
\end{remark}

\begin{example}
Let $\xi$ be a real random variable with mean zero, variance one, and $\E|\xi^4| < \infty$.  Let $A_n$ be an $n \times N$ matrix where each entry is an i.i.d. copy of $\xi$.  If $N/n \rightarrow c \in (0, \infty)$, then $A_n$ satisfies Definition \ref{def:C0}.  All the results in this paper are already known for such matrices with i.i.d. entries.  See for example \cite{mp}, \cite[Chapter 3]{bai-book}, \cite{B}, and references contained therein.  
\end{example}

\begin{example} \label{example:bern}
Let $A_n$ be a $n \times (2n)$ matrix where the rows are i.i.d. random vectors such that the entries of $r_k^{(n)}$ are $\pm 1$ symmetric Bernoulli random variables chosen uniformly such that the sum of the entries of each row is zero.  Then the sequence $\{A_n\}_{n \geq 1}$ obeys condition {\bf C0}.  Indeed, one can compute 
\begin{align*}
	\E[ \zeta_{ij}^{(n)}] &= 0,\\ \var [\zeta_{ki}^{(n)}] &= 1, \\
	\E[ \zeta_{ki}^{(n)} \zeta_{kj}^{(n)} ] &= - \frac{1}{2N-1} \text{ for } i \neq j,
	\intertext{and}
	\E[ \zeta_{ki}^{(n)} \zeta_{kj}^{(n)} \zeta_{kl}^{(n)} \zeta_{km}^{(n)}] &= \frac{12N^2 - 12N}{2N(2N-1)(2N-2)(2N-3)} = O\left( \frac{1}{N^2} \right)
\end{align*}
for $i,j,l,m$ distinct, where $N = 2n$.  In particular, one finds that $\gamma_n = n^{-1/2}$ and $q_n, \rho_n, \beta_n = 0$.  
\end{example}

Let us mention that the conditions in Definition \ref{def:C0} are similar to the assumptions of Theorem 1 in \cite{mp}.  However, in \cite{mp}, the authors require the rows of $A_n$ to be independent.   

Also, the sequence of random matrices defined in Example \ref{example:bern} satisfies condition {\bf C0}, but does not satisfy the assumptions of the theorems provided in \cite{aubrun}, \cite{PP}, \cite{GT}, or \cite{A}.  

Let $\|M\|$ denote the spectral norm of the matrix $M$.  In this paper, we shall prove the following theorems.  

\begin{theorem} \label{thm:main}
Let $\{A_n\}_{n \geq 1}$ be a sequence of real random matrices that obey condition {\bf C0} and assume $c_n := N/n \rightarrow c \in (0, \infty)$.  Then 
$$ \|\E F^{\frac{1}{n}A_n A_n^\mathrm{T}} - F_{c} \| : = \sup_{x} |\E F^{\frac{1}{n}A_n A_n^\mathrm{T}} - F_{c}| \longrightarrow 0 $$
as $n \rightarrow \infty$.  Moreover, if there exists $p>1$ such that
\begin{equation} \label{eq:betasum}
	\sum_{n=1}^\infty \frac{ (\beta_n+1)^p }{n^{p/2}} < \infty 
\end{equation}
then
$$ \|F^{\frac{1}{n}A_n A_n^\mathrm{T}} - F_{c} \| \longrightarrow 0 $$
almost surely as $n \rightarrow \infty$.  
\end{theorem}

\begin{theorem} \label{thm:rate}
Let $\{A_n\}_{n \geq 1}$ be a sequence of real random matrices that obey condition {\bf C0} and assume $c_n := N/n  \geq 1$ such that $c_n \rightarrow c \in [1, \infty)$.  Additionally assume that 
\begin{equation} \label{eq:norm_bound}
	\limsup_{n \rightarrow \infty} \frac{1}{n} \E \| A_n A_n^\mathrm{T} \| < \infty. 
\end{equation}
Then we obtain that
$$ \|\E F^{\frac{1}{n} A_n A_n^\mathrm{T}} - F_{c_n} \| = O \left( \max\left( q_n^{1/22}, \gamma_n^{1/22}, \rho_n^{1/22}, \left( \frac{(\beta_n+1)^2}{n} \right)^{1/22} \right) \right). $$
\end{theorem}

\begin{remark} \label{rem:assumption}
We stated Theorem \ref{thm:main} for a sequence of random matrices that obey condition {\bf C0}.  However, it is actually possible to prove the convergence of the expected ESD without condition \eqref{C0:indep} from Definition \ref{def:C0}.  That is, if the sequence $\{A_n\}_{n \geq 1}$ satisfies conditions \eqref{C0:uncor} - \eqref{C0:row_variance} from Definition \ref{def:C0} with $c_n \rightarrow c \in (0, \infty)$, then 
$$ \|\E F^{\frac{1}{n}A_n A_n^\mathrm{T}} - F_{c} \| \longrightarrow 0 $$
as $n \rightarrow \infty$.  The proof of this statement repeats the proof of Theorem \ref{thm:main} almost exactly.  We detail the necessary changes in Remark \ref{rem:changes}.  It should be noted that the almost sure convergence portion of Theorem \ref{thm:main} still requires condition \eqref{C0:indep} from Definition \ref{def:C0} and \eqref{eq:betasum}.  
\end{remark}

\begin{remark}
Without any additional information on the convergence rate of $c_n$ to $c$, we cannot obtain a rate of convergence of $\| \E F^{A_n A_n^\mathrm{T}} - F_{c} \|$.  This is why $F_{c_n}$ appears in Theorem \ref{thm:rate}.
\end{remark}

\begin{remark}
The rates obtained in Theorem \ref{thm:rate} are not optimal and are obtained as a simple corollary to Lemma \ref{lemma:equation} below.  
\end{remark}

\begin{example}
Let $\{A_n\}_{n \geq 1}$ be the sequence of random matrices defined in Example \ref{example:bern}.  Theorem \ref{thm:main} implies that 
$$ \|F^{\frac{1}{n}A_n A_n^\mathrm{T}} - F_{2} \| \longrightarrow 0 $$
almost surely as $n \rightarrow \infty$.  We will now use Theorem \ref{thm:rate} to obtain a rate of convergence for $\E F^{\frac{1}{n} A_n A_n^\mathrm{T}}$.  We must verify that \eqref{eq:norm_bound} holds.  By \cite[Theorem 3.13]{ALPT} 
\footnote{One technical assumption required by Theorem 3.13 is control over the $\psi_1$-norm ($\| \cdot \|_{\psi_1}$) of the term $|\langle \xi,y \rangle|$ where $\xi$ is a row of the matrix $A_n$ and $y$ is an arbitrary unit vector.  In particular, one can show that $$ \| \langle \xi,y \rangle \|_{\psi_1} \leq \frac{\sqrt{n}}{\log n^{1+\epsilon}}. $$  The bound follows by applying Markov's inequality, which yields $$ \Prob( |\langle \xi,y \rangle| > t) = O(t^{-4}), $$ and then taking $t = n^{1/3}$.  },
 there exists $C,C'>0$ such that for any $0 < \epsilon < 1/3$,
$$ \Prob \left( \|A_n A_n^\mathrm{T} \| \geq C n \right) \leq C'\frac{\log n}{n^{1+\epsilon}}. $$ 
Since we always have the bound
$$ \| A_n A_n ^\mathrm{T} \| \leq \tr (A_n A_n^\mathrm{T}) = n^2, $$
it follows that
$$ \E \| A_n A_n ^\mathrm{T} \| = O(n). $$
Therefore, Theorem \ref{thm:rate} gives the rate of convergence
$$ \|\E F^{\frac{1}{n} A_n A_n^\mathrm{T}} - F_{2} \| = O(n^{-1/44}). $$
\end{example}

\section{Stieltjes Transform}

If $G(x)$ is a function of bounded variation on the real line, then its Stieltjes transform is defined by
$$ S_G (z) = \int \frac{1}{x - z} dG(x) $$
for $z \in D := \left\{ z \in \C : \Im(z) > 0 \right \}$.  

Let $m_c(z)$ be the Stieltjes transform of $F_c$, the distribution function of the Marchenko-Pastur law with parameter $c$.  One can then check (see for example \cite{bai-book}), that
$$ m_c(z)  = \frac{ -(z+1-c) + \sqrt{(z+1-c)^2 - 4z} }{2z}. $$
Furthermore, $m_c(z)$ can be characterized uniquely as the solution to 
\begin{equation} \label{eq:mcz}
	m_c(z) = \frac{1}{c - 1 - z -z m_c(z)} 
\end{equation}
that satisfies $\Im(z m_c(z)) \geq 0$ for all $z$ with $\Im z > 0$.  

We will study the Stieltjes transform of the ESD of the random matrix $\frac{1}{n}A_n A_n^\mathrm{T}$ in order to prove Theorems \ref{thm:main} and \ref{thm:rate}.  In particular, the following lemma states that it suffices to show the convergence of the Stieltjes transform of the ESD to the Stieltjes transform of $F_c$.

\begin{lemma}[{\cite[Theorem B.9]{bai-book}}]
Assume that $\{G_n\}$ is a sequence of functions of bounded variation and $G_n(-\infty)=0$ for all $n$.  Then
$$ \lim_{n \rightarrow \infty} s_{G_n}(z) = s(z) \quad \forall z \in D $$
if and only if there is a function of bounded variation $G$ with $G(-\infty) = 0$ and Stieltjes transform $s(z)$ and such that $G_n \rightarrow G$ vaguely.  
\end{lemma}

We will also use the following lemma in order to establish the rate of convergence in Theorem \ref{thm:rate}.

\begin{lemma} [ {\cite[Theorem B.14]{bai-book}} ] \label{lemma:bai}
Let $F$ be a distribution function and let $G$ be a function of bounded variation satisfying $\int|F(x) - G(x)|dx < \infty$.  Denote their Stieltjes transforms by $s_F(z)$ and $s_G(z)$ respectively, where $z = u+iv \in D$.  Then
\begin{align*} 
	\|F-G \| &\leq \frac{1}{\pi (1 - \xi)(2 \rho - 1)} \Bigg( \int_{-A}^A |s_F(z) - s_G(z)|du \\
	& \qquad + 2 \pi v^{-1} \int_{|x|>B} |F(x) - G(x)|dx \\
	& \qquad + v^{-1} \sup_x \int_{|u|\leq 2va} |G(x+u) - G(x)|du \Bigg), 
\end{align*}
where the constants $A > B > 0$, $\xi$, and $a$ are restricted by $\rho = \frac{1}{\pi} \int_{|u| \leq a} \frac{1}{u^2 + 1} du > \frac{1}{2}$, and $\xi = \frac{4B}{\pi(A-B)(2 \rho - 1)} \in (0,1)$.  
\end{lemma}

\section{Proof of Theorems \ref{thm:main} and \ref{thm:rate} }

Let $\{ A_n \}_{n \geq 1}$ be a sequence of real random matrices that obeys condition \textbf{C0} and assume $c_n := N/n \rightarrow c \in (0, \infty)$.  We begin by introducing some notation.  Let 
$$ s_n(z) := \frac{1}{n} \tr \left( \frac{1}{n} A_n A_n^\mathrm{T} - z I_n \right)^{-1} =  \frac{1}{n} \tr \left( \frac{1}{n} A_n A_n^\mathrm{T} - z \right)^{-1}$$
where $I_n$ is the identity matrix of order $n$ and $z = u + iv$.  Fix $\alpha > 0$ and let 
$$ D_{\alpha, n} := \{ z = u + iv \in \C : |u| \leq \alpha, v_n \leq v \leq 1 \}$$
where $v_n$ is a sequence we will choose later such that $0 < v_n < 1$ for all $n$.  We will eventually allow the sequence $v_n$ to approach zero as $n$ tends to infinity.  

We will use the following lemma to prove Theorems \ref{thm:main} and \ref{thm:rate}.  

\begin{lemma} \label{lemma:equation}
Suppose $\{ A_n \}_{n \geq 1}$ is a sequence of real random matrices that obey condition \textbf{C0} and assume $c_n := N/n \rightarrow c \in (0, \infty)$.  Then for any $\alpha > 0$
$$ \sup_{z \in D_{\alpha,n}} \left| \E s_n(z) - \frac{1}{ c_n - 1 -z -z\E s_n(z)} \right| = O \left( \frac{\sqrt{q_n} + \sqrt{\gamma_n} + \sqrt{\rho_n}}{v^3_n} + \frac{\beta_n + 1}{\sqrt{n} v^3_n} \right). $$
\end{lemma}

We prove Lemma \ref{lemma:equation} in Section \ref{sec:lemma}.  For the moment, assume this lemma.  First, take $v_n > 0$ to be fixed.  Then $D_{\alpha,n}$ does not change with $n$.  Since $c_n \rightarrow c$, we obtain that
$$ \E s_n(z) = \frac{1}{ c - 1 -z -z\E s_n(z)} + o(1) $$
for all $z \in D_{\alpha, n}$.  Fix $z_0 = u_0 + iv_0 \in D_{\alpha, n}$.  Since $|\E s_n(z_0)| \leq \frac{1}{v_0}$, one can use a compactness argument to obtain a convergent subsequence $\E s_{n_k}(z_0) \rightarrow s(z_0)$.  Then $s(z_0)$ must satisfy the equation
$$ s(z_0) = \frac{1}{ c - 1 -z_0 -z_0 s(z_0)}. $$
Also, since the eigenvalues of $A_n A_n^\mathrm{T}$ are non-negative, $\Im(z \E s_n(z)) \geq 0$ for all $\Im(z) >0$ and hence $\Im(z_0 s(z_0)) \geq 0$.  Thus, by the characterization \eqref{eq:mcz}, it follows that $s(z_0) = m_c(z_0)$.  Since every convergent subsequence of $\{\E s_n(z_0) \}$ must converge to the same limit, we obtain that
$$ \lim_{n \rightarrow \infty} \E s_n(z_0) = m_c(z_0) $$
and since $z_0 \in D_{\alpha,n}$ was arbitrary, one obtains
\begin{equation} \label{eq:smc}
	\lim_{n \rightarrow \infty} \E s_n(z) = m_c(z)
\end{equation}
for all $z \in D_{\alpha,n}$.  Finally, since $|\E s_n(z)| \leq \frac{1}{v}$, Vitali's Convergence Theorem implies that \eqref{eq:smc} holds for all $z \in D$.  Therefore,
$$ \|\E F^{\frac{1}{n}A_n A_n^\mathrm{T}} - F_{c} \| : = \sup_{x} |\E F^{\frac{1}{n}A_n A_n^\mathrm{T}} - F_{c}| \longrightarrow 0 $$
as $n \rightarrow \infty$.  

To obtain the almost sure convergence in Theorem \ref{thm:main}, one repeats the argument above and then applies the Borel-Cantelli lemma, since
$$ \Prob \left( |s_n(z) - \E s_n(z) | \geq \epsilon \right) \leq \frac{C_p (\beta_n + 1)^p}{n^{p/2} v_n^p \epsilon^p} $$
by Lemma \ref{lemma:variance} from Appendix \ref{app:variance}.  

To prove Theorem \ref{thm:rate}, we will apply Lemma \ref{lemma:bai}.  Under assumption \eqref{eq:norm_bound}, there exists $B>0$ such that
$$ \E F^{\frac{1}{n}A_n A_n^\mathrm{T}}(x) - F_{c_n}(x) = 0 $$
for all $|x| > B$ and $n$ sufficiently large.  

By \cite[Lemma 8.15]{bai-book}, it follows that
$$ \sup_x \int_{|u| < v_n} |F_{c_n}(x+u) - F_{c_n}(x)| = O \left( v_n^{3/2} \right) $$
for $c_n \geq 1$.  

From Lemma \ref{lemma:equation}, we have that
$$ \E s_n(z) = \frac{1}{c_n - 1 - z -z \E s_n(z)} + \delta_n $$
for all $z \in D_{\alpha,n}$.  Thus
$$ z (\E s_n(z))^2 + \E s_n(z) (z + 1 - c_n) + 1 = \delta_n (1- c_n - z - z \E s_n(z)). $$
By subtracting the quadratic equation for $m_{c_n}(z)$ obtained from \eqref{eq:mcz}, one finds that
$$ |\E s_n(z) - m_{c_n}(z)| = \frac{|\delta_n||1- c_n + z + z \E s_n(z)|}{|z \E s_n(z) + z m_{c_n}(z) + z + 1 - c_n|} = O \left( \frac{\delta_n}{v_n^2} \right). $$

Therefore, from Lemma \ref{lemma:bai}, one obtains that
$$ \|\E F^{\frac{1}{n}A_n A_n^\mathrm{T}} - F_{c_n} \| = O \left( \frac{\sqrt{q_n} + \sqrt{\gamma_n} + \sqrt{\rho_n}}{v^5_n} + \frac{\beta_n+1}{\sqrt{n} v^5_n} + \sqrt{v_n} \right) $$
and hence we can take
$$ v_n = \max\left( q_n^{1/11}, \gamma_n^{1/11}, \rho_n^{1/11}, \left( \frac{(\beta_n+1)^2}{n} \right)^{1/11}  \right). $$
The proof of Theorem \ref{thm:rate} is complete.  

It only remains to prove Lemma \ref{lemma:equation}.  

\section{Proof of Lemma \ref{lemma:equation}} \label{sec:lemma}

Let $\{ A_n \}_{n \geq 1}$ be a sequence of real random matrices that obey condition \textbf{C0} and assume $c_n := N/n \rightarrow c \in (0, \infty)$.  Fix $\alpha > 0$.  In order to simplify notation, we drop the superscript $(n)$ and write $\zeta_{ij}$ and $r_k$ for the entries of $A_n$ and the rows of $A_n$ respectively.  We define the resolvent 
$$ R_n(z) :=  \left( \frac{1}{n} A_n A_n^\mathrm{T} - z \right)^{-1}.  $$
Using the Schur complement, we obtain that
$$ (R_n(z))_{kk} = \frac{1}{ \frac{1}{n} |r_k|^2 - z - \frac{1}{n}r_k A_{n,k}^\mathrm{T} R_{n,k}(z) A_{n,k} r_k^\mathrm{T} } =: \frac{1}{a_k} $$
where $A_{n,k}$ is obtained from the matrix $A_n$ by removing the $k$-th row and 
$$ R_{n,k} = \left( \frac{1}{n} A_{n,k} A_{n,k}^\mathrm{T} - z \right)^{-1}. $$

Since $| (R_n(z))_{kk}| \leq \|R_{n}(z)\| \leq \frac{1}{v_n}$, we obtain that $|a_k| \geq v_n$.  Thus,
\begin{align} \label{eq:san}
	\left| \E s_n(z) - \frac{1}{n} \sum_{k=1}^n \frac{1}{ \E a_k} \right| \leq \frac{1}{n v_n^2} \sum_{k=1}^n \E| a_k - \E a_k |. 
\end{align}

We now compute the expectation of $a_k$.  By condition \eqref{C0:var} in Definition \ref{def:C0}, we have that
$$ \sup_{k} \left| \E \frac{1}{n} |r_k|^2 - c_n \right| \leq q_n. $$
For convenience, write
$$ B_{n,k} := A_{n,k}^\mathrm{T} R_{n,k}(z) A_{n,k}. $$
We first note that $\sup_{k} \|B_{n,k}\| = O(v_n^{-1})$.  Indeed, since $|z|^2 \leq \alpha^2 + 1 = O(1)$, 
$$ \|B_{n,k}\| = \|  R_{n,k}(z) A_{n,k} A_{n,k}^\mathrm{T} \| = \| I_{n-1} + z R_{n,k}(z) \| \leq 1 + \frac{|z|}{v_n} $$
for all $k= 1, \ldots, n$.  Then we have that
\begin{align} \label{eq:rbr}
	\E \frac{1}{n} r_k B_{n,k} r_k^\mathrm{T} = \frac{1}{n} \sum_{i,j=1}^N \E \left[ \E_k[\zeta_{ki} \zeta_{kj}] (B_{n,k})_{ij} \right] = \frac{1}{n} \E \tr B_{n,k} + \epsilon_{n,k} + O\left( \frac{q_n}{v_n} \right)
\end{align}
uniformly for all $k$ (by condition \eqref{C0:var} in Definition \ref{def:C0}) where 
$$ \epsilon_{n,k} := \frac{1}{n} \sum_{i \neq j} \E \left[ \E_k[\zeta_{ki} \zeta_{kj}] (B_{n,k})_{ij} \right]. $$
By condition \eqref{C0:2cor}, we have that
\begin{align*}
 |\epsilon_{n,k}| &\leq \left( \E \frac{\gamma_n^2}{n^3} \sum_{i, j,s,t=1}^N |(B_{n,k})_{ij}| |(B_{n,k})_{s,t}| \right)^{1/2} \\
	& \leq \left( \E \frac{2 \gamma_n^2}{n} \tr (B_{n,k} B_{n,k}^\ast) \right)^{1/2} = O \left( v_n^{-1} \gamma_n \right) 
\end{align*}
uniformly in $k$.  

Combining the above yields,
\begin{equation} \label{eq:acz}
	\sup_{k} \left| \E a_k - \left( c_n - z - \E \frac{1}{n} \tr B_{n,k}\right) \right| = O\left( \frac{q_n + \gamma_n }{v_n} \right)
\end{equation}
We now note that $\frac{1}{n}\tr B_{n,k} = \frac{n-1}{n} + z \frac{1}{n} \tr R_{n,k}(z)$.  By equation (3.11) in \cite{bai} (or alternatively, by Cauchy's Interlacing Theorem), one finds that 
\begin{equation} \label{eq:cauchy}
	\left| \frac{1}{n} \tr R_{n,k}(z) - \frac{1}{n} \tr R_n(z) \right| = O \left( \frac{1}{n v_n} \right)
\end{equation}
uniformly in $k$.  Therefore, from \eqref{eq:acz} and the fact that $\frac{1}{n} \tr R_n(z) = s_n(z)$, we obtain that
\begin{equation} \label{eq: }
	\sup_{k} \left| \E a_k - \left( c_n - 1 - z - z \E s_n(z) \right) \right| = O\left( \frac{q_n + \gamma_n }{v_n} + \frac{1}{n v_n } \right)
\end{equation}

We now turn our attention to obtaining a bound for $\E| a_k - \E a_k |$.  First we note that
$$ \E \left| \frac{1}{n} |r_k|^2 - c_n \right| \leq q_n $$
by condition \eqref{C0:var} of Definition \ref{def:C0}.  Using \eqref{eq:rbr} and the bounds obtained above for $\epsilon_{n,k}$, we have that
\begin{align*}
	\E \left| \frac{1}{n}r_k  B_{n,k} r_k^\mathrm{T} - \E \frac{1}{n} \tr B_{n,k} \right|^2 &= \frac{1}{n^2} \sum_{i,j,s,t=1}^N \E[ \zeta_{ki} \zeta_{kj} \zeta_{ks} \zeta_{kt} (B_{n,k})_{ij} (B_{n,k})_{st}] \\
	& \qquad  - \left( \E \frac{1}{n} \tr B_{n,k} \right)^2 + O \left( \frac{q_n + \gamma_n}{v_n^2} \right). 
\end{align*}
For the sum
$$ \frac{1}{n^2} \sum_{i,j,s,t=1}^N \E[ \zeta_{ki} \zeta_{kj} \zeta_{ks} \zeta_{kt} (B_{n,k})_{ij} (B_{n,k})_{st}]. $$
we consider several cases:
\begin{enumerate}[(a)]
\item When we sum over all $i,j,s,t$ distinct, one finds
$$ \frac{1}{n^2} \sum \E[ \zeta_{ki} \zeta_{kj} \zeta_{ks} \zeta_{kt} (B_{n,k})_{ij} (B_{n,k})_{st}] = O \left( \frac{\gamma_n}{v_n^2} \right) $$
by condition \eqref{C0:4cor}.  
\item When we sum of all $i=j,s=t$, we obtain
$$ \frac{1}{n^2} \sum \E[ \zeta_{ki} \zeta_{kj} \zeta_{ks} \zeta_{kt} (B_{n,k})_{ij} (B_{n,k})_{st}] = \E \left[ \left( \frac{1}{n} \tr B_{n,k} \right)^2 \right] + O \left( \frac{\rho_n}{v_n^2} \right) $$
by condition \eqref{C0:row_variance}.  
\item When we sum over all $i=s,j=t$ (or $i=t,j=s$), we have
\begin{align*}
	\frac{1}{n^2} &\sum \E[ \zeta_{ki} \zeta_{kj} \zeta_{ks} \zeta_{kt} (B_{n,k})_{ij} (B_{n,k})_{st}] \\
	& \qquad = \E \left[ \frac{1}{n^2} \tr(B_{n,k} B_{n,k}^\ast) \right] + O\left( \frac{\rho_n}{v_n^2} \right) \\
	&\qquad = O \left( \frac{\rho_n}{v_n^2} + \frac{1}{nv_n^2} \right)
\end{align*}
by condition \eqref{C0:row_variance}.  
\item When $i=s, j \neq t$ (or $i=t,j \neq s$), one finds that
$$ \frac{1}{n^2} \sum \E[ \zeta_{ki} \zeta_{kj} \zeta_{ks} \zeta_{kt} (B_{n,k})_{ij} (B_{n,k})_{st}] = O \left( \frac{\gamma_n}{\sqrt{n} v_n^2} \right) $$
by condition \eqref{C0:2cor}.  
\item When $i=j,s \neq t$, we find
\begin{align*}
	\frac{1}{n^2} &\sum \E[ \zeta_{ki} \zeta_{kj} \zeta_{ks} \zeta_{kt} (B_{n,k})_{ij} (B_{n,k})_{st}] \\
	& \qquad \leq \frac{\gamma_n}{n^{3/2}v_n} \E\left( \sum_{s,t} |(B_{n,k})_{st}| \right) \\
	& \qquad \leq \frac{\gamma_n}{n^{3/2}v_n} \E\left( \sum_{s,t,l,m} |(B_{n,k})_{st}| |(B_{n,k})_{lm}| \right)^{1/2} \\	& \qquad \leq \frac{\gamma_n}{n^{3/2}v_n} n \E \left( 2 \tr (B_{n,k} B_{n,k}^\ast) \right)^{1/2} = O\left( \frac{\gamma_n}{v_n^2} \right)
\end{align*}
by Cauchy-Schwarz and condition \eqref{C0:2cor}.  
\end{enumerate}

Therefore, we obtain that
\begin{align} \label{eq:erb}
	\E \left| \frac{1}{n}r_k  B_{n,k} r_k^\mathrm{T} - \E \frac{1}{n} \tr B_{n,k} \right|^2 = \var\left( \frac{1}{n} \tr B_{n,k} \right) + O \left( \frac{q_n + \gamma_n + \rho_n}{v_n^2} + \frac{1}{n v_n^2} \right).
\end{align}
The bound in \eqref{eq:erb} holds uniformly in $k$ since the bounds in conditions \eqref{C0:2cor}, \eqref{C0:4cor}, and \eqref{C0:row_variance} of Definition \ref{def:C0} hold uniformly in $k$.  

By Lemma \ref{lemma:variance} in Appendix \ref{app:variance}, we have that
$$ \sup_{k} \var\left( \frac{1}{n} \tr B_{n,k} \right) = O \left( \frac{(\beta_n+1)^2}{n v_n^2} \right) $$
and hence
\begin{align}
	\E \left| \frac{1}{n}r_k  B_{n,k} r_k^\mathrm{T} - \E \frac{1}{n} \tr B_{n,k} \right|^2 = O \left( \frac{q_n + \gamma_n + \rho_n}{v_n^2} + \frac{(\beta_n+1)^2}{n v_n^2} \right).  
\end{align}
Therefore
\begin{align}
	\sup_{k} \E|a_k - \E a_k| = O \left( \frac{\sqrt{q_n} + \sqrt{\gamma_n} + \sqrt{\rho_n}}{v_n} + \frac{\beta_n+1}{\sqrt{n} v_n} \right).  
\end{align}
One can also observe that $\Im (zs_n(z)) \geq 0$ for all $z$ with $\Im z > 0$, since the eigenvalues of $A_n A_n^\mathrm{T}$ are non-negative.  Combining this fact with equations \eqref{eq:san} and \eqref{eq:acz} and the estimates above, we obtain that
$$ \left| \E s_n(z) - \frac{1}{1 - c_n - z - z s_n(z)} \right| = O \left( \frac{\sqrt{q_n} + \sqrt{\gamma_n} + \sqrt{\rho_n}}{v^3_n} + \frac{\beta_n+1}{\sqrt{n} v^3_n} \right) $$
where the bound holds uniformly for $z \in D_{\alpha,n}$.  The proof of Lemma \ref{lemma:equation} is complete.

\begin{remark} \label{rem:changes}
As noted in Remark \ref{rem:assumption}, it is possible to show that if the sequence $\{A_n\}_{n \geq 1}$ satisfies conditions \eqref{C0:uncor} - \eqref{C0:row_variance} from Definition \ref{def:C0} with $c_n \rightarrow c \in (0,\infty)$, then 
\begin{equation} \label{eq:rem:conv}
	\|\E F^{\frac{1}{n}A_n A_n^\mathrm{T}} - F_{c} \| \longrightarrow 0 
\end{equation}
as $n \rightarrow \infty$.  The proof of the above statement repeats the proof of Theorem \ref{thm:main} almost exactly; we now detail the necessary changes.  

Since the Stieltjes transform is an analytic and bounded function, it suffices to prove the convergence of $\E s_n(z)$ to $m_c(z)$ for all $z$ in a compact set in the upper-half plane with $\Im(z) \geq \kappa$ for a sufficiently large constant $\kappa$ to be chosen later.   

A careful reading of the proof of Lemma \ref{lemma:equation} reveals that condition \eqref{C0:indep} from Definition \ref{def:C0} is only used to invoke Lemma \ref{lemma:variance} and obtain the variance bound \eqref{eq:stvar}.  Thus, in order to prove \eqref{eq:rem:conv}, it suffices to show that
\begin{equation} \label{eq:rem:varshow}
	\var \left( \frac{1}{n} \tr R_n(z) \right) = o(1) 
\end{equation}
for all $z$ in a compact set in the upper-half plane with $\Im(z) \geq \kappa$.  

We decompose 
\begin{align*}
	(R_n(z))_{kk} &= \frac{1}{ \frac{1}{n} |r_k|^2 - z - \frac{1}{n} r_k A_{n,k}^\mathrm{T} R_{n,k}(z) A_{n,k} r_k^\mathrm{T} } \\
	&= \frac{1}{c_n - 1 - z - z \E s_n(z) - \epsilon_k}
\end{align*}
where
$$ \epsilon_k = c_n -1 - z\E s_n(z) - \frac{1}{n} |r_k|^2 + \frac{1}{n} r_k A_{n,k}^\mathrm{T} R_{n,k}(z) A_{n,k} r_k^\mathrm{T}. $$
Thus 
\begin{align*}
	(R_{n}(z))_{kk} &= \frac{1}{c_n - 1 -z -z \E s_n(z)} \\
	& \qquad + \frac{\epsilon_k}{ (c_n - 1 - z - z \E s_n(z))((c_n - 1 - z - z \E s_n(z) - \epsilon_k)} \\
	& = \frac{1}{c_n - 1 - z -z\E s_n(z)} [1 + (R_{n}(z))_{kk} \epsilon_k]. 
\end{align*}
Taking $\Im(z) \geq \kappa$ we obtain
$$ \var \left( \frac{1}{n} \tr R_{n}(z) \right) \leq \frac{C}{\kappa^4 n} \sum_{k=1}^n \E|\epsilon_k|^2 $$
for some absolute constant $C>0$.  Using condition \eqref{C0:row_variance} from Definition \ref{def:C0}, \eqref{eq:cauchy}, and \eqref{eq:erb}, we bound $\E|\epsilon_k|^2$ and obtain
$$ \var \left( \frac{1}{n} \tr R_{n}(z) \right) \leq \frac{C|z|^2}{\kappa^4} \var \left(\frac{1}{n} \tr R_{n}(z) \right) + O \left(q_n + \gamma_n + \rho_n + \frac{1}{\sqrt{n}} \right).  $$
Taking $z$ in a compact set for which $|z|^2/\kappa^4$ is sufficiently small verifies \eqref{eq:rem:varshow}.  
\end{remark}

\appendix 
\section{Estimate of Variance of Stieltjes Transform} \label{app:variance}

\begin{lemma} \label{lemma:variance}
Let $A_n$ be an $n \times N$ real random matrix with rows $r^{(n)}_1, \ldots, r^{(n)}_n$ that satisfy condition \eqref{C0:indep} from Definition \ref{def:C0}.  Then for every $p>1$ there exists a constant $C_p > 0$ (depending only on $p$) such that 
\begin{equation} \label{eq:st4mom}
	\E \left| \frac{1}{n} \tr \left( \frac{1}{n} A_n A_n^\mathrm{T} - zI_n \right)^{-1} - \frac{1}{n} \E\tr \left( \frac{1}{n} A_n A_n^\mathrm{T} - zI_n \right)^{-1}\right|^p \leq \frac{C_p (\beta_n + 1)^p }{n^{p/2} |\Im z|^p}
\end{equation}
for any $z$ with $\Im z \neq 0$.  In particular, there exists an absolute constant $C>0$ such that
\begin{equation} \label{eq:stvar}
	\var \left[ \frac{1}{n} \tr \left( \frac{1}{n} A_n A_n^\mathrm{T} - zI_n \right)^{-1} \right] \leq \frac{C (\beta_n+1)^2}{n |\Im z|^2}
\end{equation}
for any $z$ with $\Im z \neq 0$.
\end{lemma}

\begin{proof}
Since \eqref{eq:st4mom} implies \eqref{eq:stvar} when $p=2$, it suffices to prove \eqref{eq:st4mom} for arbitrary $p>1$.  Let $\E_{\leq k}$ denote the conditional expectation with respect to the $\sigma$-algebra generated by $r^{(n)}_1, \ldots, r^{(n)}_k$.  Define
$$ R_n(z) := \left( \frac{1}{n} A_n A_n^\mathrm{T} - zI_n \right)^{-1} $$
and
$$ Y_k := \E_{\leq k} \frac{1}{n} \tr R_n(z)$$
for $k = 0, 1, \ldots, n$.  Then $\{Y_k\}_{k=0}^n$ is a martingale since $\E_{\leq k} Y_{k+1} = Y_k$.  Define the martingale difference sequence
$$ \alpha_k := Y_{k} - Y_{k-1} $$
for $k = 1, 2, \ldots, n$.  We then note that
\begin{equation} \label{eq:kak}
	\sum_{k=1}^n \alpha_k = \frac{1}{n} \tr R_n(z) - \E \frac{1}{n} \tr R_n(z). 
\end{equation}

We will bound the $p$-th moment of the sum in \eqref{eq:kak}, but first we obtain a bound on the individual summands $\alpha_k$.  

For each $1 \leq k \leq n$, define the set
$$ J(k) := \{ 1 \leq j \leq n : k \leq j \leq k+ \beta_n \}. $$
Now let $A_{n,J(k)}$ be obtained from the matrix $A_n$ by removing row $j$ if and only if $j \in J(k)$.  Let  
$$ R_{n,J(k)}(z) = \left( \frac{1}{n} A_{n,J(k)} A_{n,J(k)}^\mathrm{T} - zI_n \right)^{-1}. $$

A simple computation (see for instance \cite[Example 5.1.5]{Dur}) reveals that $\E_{\leq k} \tr R_{n,J(k)}(z) = \E_{\leq k-1} \tr R_{n,J(k)}(z)$ by condition \eqref{C0:indep} of Definition \ref{def:C0}.  Thus 
$$ \alpha_k = \E_{\leq k} \frac{1}{n} \left( \tr R_n(z) - \tr R_{n,J(k)}(z) \right) - \E_{\leq k-1} \frac{1}{n} \left( \tr R_n(z) - \tr R_{n,J(k)}(z) \right). $$

Using the triangle inequality and equation (3.11) in \cite{bai}, we have that
$$ \left| \tr R_n(z) - \tr R_{n,J(k)}(z) \right| \leq \frac{\beta_n + 1}{| \Im z|} $$
and hence 
$$ |\alpha_k| \leq \frac{2 (\beta_n + 1)}{n|\Im z|}. $$

We now apply the Burkholder inequality (see for example \cite[Lemma 2.12]{bai-book} for a complex-valued version of the Burkholder inequality) and obtain that there exists a constant $C_p>0$ such that
\begin{align*}
	\E \left| \sum_{k=1}^n \alpha_k \right|^p &\leq C_p \E \left( \sum_{k=1}^n |\alpha_k|^2 \right)^{p/2}  \\
	& \leq C_p \left( \frac{4 (\beta_n + 1)^2 n}{n^2 |\Im z|^2} \right)^{p/2} \\
	& \leq \frac{C_p 2^p (\beta_n + 1)^p}{n^{p/2} |\Im z|^p}.
\end{align*}
\end{proof}

\subsection*{Acknowledgment}
The author is grateful to A. Litvak for pointing out Theorem 3.13 in \cite{ALPT}.  The author would also like to thank A. Soshnikov and D. Renfrew for useful conversations and comments.

\end{document}